\documentclass[12pt]{amsart}
\usepackage{amscd,amssymb,amsmath,multicol}
\newtheorem{thm}[equation]{Theorem}
\numberwithin{equation}{section}
\newtheorem{cor}[equation]{Corollary}

\newtheorem{lem}[equation]{Lemma}

\newtheorem{defin}[equation]{Definition}

\newtheorem{prop}[equation]{Proposition}

\begin{document}
\raggedbottom \voffset=-.7truein \hoffset=0truein \vsize=8truein
\hsize=6truein \textheight=8truein \textwidth=6truein
\baselineskip=18truept
\def\mapright#1{\ \smash{\mathop{\longrightarrow}\limits^{#1}}\ }
\def\ss{\smallskip}
\def\ssum{\sum\limits}
\def\dsum{{\displaystyle{\sum}}}
\def\la{\langle}
\def\ra{\rangle}
\def\on{\operatorname}
\def\o{\on{od}}
\def\a{\alpha}
\def\bz{{\Bbb Z}}
\def\eps{\epsilon}
\def\br{{\bold R}}
\def\bc{{\bold C}}
\def\bN{{\bold N}}
\def\nut{\widetilde{\nu}}
\def\tfrac{\textstyle\frac}
\def\product{\prod}
\def\b{\beta}
\def\G{\Gamma}
\def\g{\gamma}
\def\zt{{\Bbb Z}_2}
\def\zth{{\bold Z}_2^\wedge}
\def\bs{{\bold s}}
\def\bg{{\bold g}}
\def\bof{{\bold f}}
\def\bq{{\bold Q}}
\def\be{{\bold e}}
\def\line{\rule{.6in}{.6pt}}
\def\xb{{\overline x}}
\def\xbar{{\overline x}}
\def\ybar{{\overline y}}
\def\zbar{{\overline z}}
\def\ebar{{\overline \be}}
\def\nbar{{\overline n}}
\def\fbar{{\overline f}}
\def\Ubar{{\overline U}}
\def\et{{\widetilde e}}
\def\ni{\noindent}
\def\ms{\medskip}
\def\ehat{{\hat e}}
\def\xhat{{\widehat x}}
\def\nbar{{\overline{n}}}
\def\minp{\min\nolimits'}
\def\N{{\Bbb N}}
\def\Z{{\Bbb Z}}
\def\Q{{\Bbb Q}}
\def\R{{\Bbb R}}
\def\C{{\Bbb C}}
\def\el{\ell}
\def\mo{\on{mod}}
\def\dstyle{\displaystyle}
\def\ds{\dstyle}
\def\Remark{\noindent{\it  Remark}}
\title
{For which $p$-adic integers $x$ can $\dstyle\sum_k\tbinom xk^{-1}$ be defined?}
\author{Donald M. Davis}
\address{Department of Mathematics, Lehigh University\\Bethlehem, PA 18015, USA}
\email{dmd1@lehigh.edu}
\date{August 1, 2012}


\maketitle
\begin{abstract} Let $f(n)=\sum_k\binom nk^{-1}$. First, we show that  $f:\N\to\Q_p$  is nowhere continuous
in the $p$-adic topology. If $x$ is a $p$-adic integer, we say that $f(x)$ is $p$-definable if $\lim f(x_j)$ exists in $\Q_p$,
where $x_j$ denotes the $j$th partial sum for $x$. We prove that $f(-1)$ is $p$-definable for all primes $p$, and if $p$ is  odd, then
$-1$ is the only element of $\Z_p-\N$ for which $f(x)$ is $p$-definable. For $p=2$, we
show that if $k$ is a positive integer, then $f(-k-1)$ is not 2-definable,
but that if the 1's in the binary expansion of $x$ are eventually very sparse, then $f(x)$ is 2-definable.

Some of our proofs require that $p$ satisfy one of two conditions. There are three small primes which do not satisfy the relevant
condition, but our theorems can be proved directly for these primes. No other prime less than 100,000,000 fails to satisfy the
conditions.
\end{abstract}
\section{Statement of results}\label{intro}
Let $\N\subset\Z_p\subset\Q_p$ denote the natural numbers (including 0), $p$-adic integers, and $p$-adic numbers, respectively, with metric $d_p(x,y)=p^{-\nu_p(x-y)}$.
Here and throughout, $\nu_p(-)$ denotes the exponent of $p$ in a rational or $p$-adic number.

The function $f:\N\to\Q_p$ defined by
$$f(n)=\sum_{k=0}^n\tbinom nk^{-1}$$
has been studied in \cite{Rock}, \cite{Sury}, \cite{YZ}, and \cite{YZ2}. Throughout this paper, $f$ will always refer to this function.
Cursory calculations suggested that perhaps this $f$ might
be continuous in the 2-adic topology and extendable over $\zt$. For example, {\tt Maple} computes
\begin{prop} If $\nu_2(m+2)=\nu_2(n+2)\ge4$ and $m,n<8000$, then
$$\nu_2(f(m)-f(n))=\nu_2(m-n)+1-2\nu_2(m+2).$$
If $n\equiv14\mod 16$ and $n<70,000$, then $\nu_2(f(n))=1-\nu_2(n+2)$.\end{prop}
If this persisted without the bounds, then $f$ would extend to a continuous function on $\zt\cap\{x:x\equiv14\ (16),\ x\ne-2\}$.
However these bounds are not large enough to reveal the problems that can occur.

For technical reasons, some of our results involve the notion of Wieferich primes.
Recall that a {\it Wieferich prime} $p$ is one for which $p^2$ divides $2^{p-1}-1$. The only known Wieferich primes are 1093 and 3511.
Let $LW$ denote the set of Wieferich primes greater than 3511.  The name of this set refers to ``large Wieferich." As of May 2012,
it was known that $LW$ contains no integers less than $17\cdot10^{15}$.(\cite{wiki})

Others of our results require that $p$ satisfies that for all $n$ such that $1\le n\le p-2$, $\nu_p(f(n))\le 1$.  We will call such primes {\it good}.
Paul Zimmermann has tested all $p<10^8$
and found that in this range there is only one case, $p=23$ and $n=12$, in which $\nu_p(f(n))>1$, with $\nu_{23}(f(12))=2$.  Thus all primes less than $100,000,000$ are good except for $p=23$.

We will prove the following theorem in Section \ref{pfsec}. Its corollary is the first of our main results.
\begin{thm}\label{2^e-k} a. If $k\ge1$ and $e>\max\{1,j+\nu_2(j):0<j<k\}$, then
$$\nu_2(f(2^e-k-1))=k+\nu_2(k)-e.$$
\begin{itemize}\item[b.] If $p$ is an odd prime which is not a Wieferich prime, $k\ge1$, and $e>\max\{1,1+\nu_p(j):0<j<k\}$, then
\begin{equation}\label{p^e-k-1}\nu_p(f(p^e-k-1))=\nu_p(k)+1-e.\end{equation}
\item[c.] If $p=1093$ or $3511$, $k\ge1$, and $e>\max\{1,2+\nu_p(j):0<j<k\}$, then
\begin{equation}\label{two}\nu_p(f(p^e-k-1))=\nu_p(k)+2-e.\end{equation}
\end{itemize}
\end{thm}

\begin{cor} If $p\not\in LW$, then the function $f:\N\to\Q_p$ is nowhere continuous (in the $p$-adic topology).\end{cor}
\begin{proof} We give the proof when $p=2$. The proof when $p$ is odd is extremely similar, using (\ref{p^e-k-1}) and (\ref{two}).

Let $n\in\N$, and let $\eps>0$ be given. We will show there exists $m\in\N$ such that $d_2(m,n)<\eps$ and $d_2(f(m),f(n))>1$.
Choose $e$ so that $2^{-e}<\eps$ and $2^e>n+1$. By Theorem \ref{2^e-k}, if $L>\max\{j+\nu_2(j):0<j<2^e-n-1\}$, then
$$\nu_2(f(2^L-(2^e-n)))=2^e-n-1+\nu_2(n+1)-L.$$
Choose $L$ large enough that this is less than $\nu_2(f(n))$ and less than 0. Call this value $-r$ with $r\ge1$. Then $\nu_2(f(2^L-(2^e-n))-f(n))=-r$ and so
$$d_2(f(2^L-(2^e-n)),f(n))=2^r\ge2,$$
while $d_2(2^L-(2^e-n),n)=2^{-e}<\eps$. So $m=2^L-(2^e-n)$ has the desired properties.\end{proof}

Since $\N$ is dense in $\Z_p$, this of course implies that $f$ cannot be extended to a function $\fbar$ that is continuous at even one point of $\Z_p$.
We state this, but omit the elementary and standard proof.
\begin{cor} If $p\not\in LW$ and $x\in\Z_p$, it is impossible to define $\fbar(x)$ so that for all sequences $\la n_i\ra$ in $\N$ such that $n_i\to x$ we have
$f(n_i)\to\fbar(x)$.\end{cor}

We remark that the summand functions $f_k(x):=\binom xk^{-1}=\frac{k!}{x(x-1)\cdots(x-(k-1))}$ are continuous on $\Z_p-\{0,\ldots,k-1\}$.

Although $\fbar(x)$ cannot be defined so that it works nicely for all sequences of positive integers approaching $x$,
it might happen that it can be defined so that it is the limit of the most natural sequence of positive integers approaching $x$,
namely the finite partial sums.
\begin{defin}\label{def} For a $p$-adic integer $x=\dstyle\sum_{i=0}^\infty\eps_ip^i$ with $\eps_i\in\{0,\ldots,p-1\}$, let $x_n=\ds\sum_{i=0}^n\eps_ip^i$. Say $\fbar(x)$ is {\bf $p$-definable} if $\lim f(x_n)$
exists in $\Q_p$. If so, define $\fbar(x)=\lim f(x_n)$.\end{defin}

The following result points to a similarity and a difference between the prime 2 and the odd primes.
\begin{thm} \label{-1} For all primes $p$, $\fbar(-1)$ is $p$-definable. If $p=2$, then $\fbar(-1)=0$, while if $p$ is odd, then $\fbar(-1)\equiv1\mod p$.\end{thm}
\begin{proof} We will prove in Proposition \ref{2^e-1} that $\nu_2(f(2^e-1)\ge2e$, from which the result for $p=2$ is immediate since $2^e-1=(-1)_{e-1}$ in the notation of Definition \ref{def}. Now let $p$ be an odd prime.
We will prove in Proposition \ref{p^e-1} that, for $e\ge1$, $\nu_p(f(p^e-1)-f(p^{e-1}-1))\ge e$. Thus $\la f(p^e-1)\ra$ is a Cauchy sequence, and since $p^e-1=(-1)_{e-1}$, we deduce that $\fbar(-1)$ is $p$-definable. Since $f(p^0-1)=1$, we obtain $f(p^e-1)\equiv1\mod p$ for all $e\ge0$.\end{proof}

We show now that, for all primes $p\not\in LW$,  $\fbar(x)$ is not $p$-definable if $x$ is an integer less than $-1$.
\begin{thm}\label{no} Let $k$ be a positive integer and $p\not\in LW$. Then $\fbar(-k-1)$ is not $p$-definable.\end{thm}
\begin{proof} We give the argument when $p$ is  odd. The argument when $p=2$ is extremely similar. Let $x=-k-1$. Let $e>\max\{1,\delta+\nu_p(j):0<j<k\}$, where $\delta=1$ if $p$ is not Wieferich, and $\delta=2$ if $p\in\{1093,3511\}$. Then $x_{e-1}=p^e-k-1$ and, by Theorem \ref{2^e-k}, $\nu_p(f(x_{e-1}))=\delta+\nu_p(k)-e$. Hence, as $e\to\infty$,
$$d_p(f(x_{e-1}),0)=p^{e-\delta-\nu(k)}\to\infty.$$
Thus $\la f(x_n)\ra$ is not a Cauchy sequence.\end{proof}

The following result applies only when $p$ is odd; the analogous statement is not true when $p=2$. It says that $-1$ is the only element of $\Z_p-\N$ for which $\fbar(x)$ is $p$-definable when $p$ is a good odd prime or 23. As noted earlier, this includes all odd primes less than 100,000,000.
\begin{thm}\label{onlyodd} If $p$ is a good odd prime or $p=23$ and $x\in\Z_p-\{-1,0,1,2,\ldots\}$, then $\fbar(x)$ is not $p$-definable.\end{thm}
\begin{proof} We will prove in Section \ref{goodsec}  that if $p$ is a good odd prime or 23 and $1\le c\le p-1$ and $0\le i<p^e-1$, then
\begin{equation}\label{pdiff}\nu_p(f(c\cdot p^e+i)-f(i))\le 0.\end{equation}
 If $x$ is as in the theorem, then there exist infinitely many $e$ for which (\ref{pdiff}) applies directly to say $\nu_p(f(x_e)-f(x_{e-1}))\le0$.
Thus $d_p(f(x_e),f(x_{e-1}))\ge 1$ and so $\la f(x_n)\ra$ is not a Cauchy sequence.\end{proof}

Our other main theorem says that $\fbar(x)$ is 2-definable for those $x$ whose infinitely many 1's are eventually very sparse.
This is a major difference between the prime 2 and the odd primes.
 This theorem will be proved in Section \ref{2sec}.
\begin{thm} Suppose $x=\ds\sum_{i=0}^\infty2^{e_i}$, $e_0<e_1<\cdots$, and there exists a positive integer $N$ such that $e_k>k+\ds\sum_{i=0}^{k-1}2^{e_i}$ for all $k\ge N$.
Then $\fbar(x)$ is 2-definable.\label{yes}\end{thm}
\ni For example,
$$\fbar(1+2^3+2^{12}+2^{2^{12}+13}+2^{2^{4109}+4110}+\cdots)$$
is 2-definable.

Since negative integers are those $x$ for which, in Definition \ref{def}, all but a finite number of $\eps_i$ equal $p-1$,
Theorems \ref{no} and \ref{yes} tell whether $\fbar(x)$ is 2-definable for $x$ at two extremes. It would be interesting to know more completely which 2-adic integers $x$ have the property that $\fbar(x)$ is 2-definable.

\section{Some $p$-exponents of $f(n)$}\label{pfsec}
In this section, we prove Theorem \ref{2^e-k} and a result, Proposition \ref{p^e-1}, which was used above and will be used again later.
We begin the proof of Theorem \ref{2^e-k} with a special case.
\begin{lem}\label{2lem} Let $p$ be a prime not in $LW$, and $e\ge2$. Then
 $$\nu_p(f(p^e-2))=\begin{cases}-(e-2)&p\in\{1093,3511\}\\
 -(e-1)&\text{otherwise.}\end{cases}$$\end{lem}
\begin{proof} Let $\a_p(n)$ denote the sum of the coefficients in the base-$p$ expansion of $n$. For $0\le i\le p^e-2$,
\begin{eqnarray*} \nu_p\tbinom{p^e-2}i&=&\tfrac 1{p-1}\bigl(\a_p(i)+\a_p(p^e-2-i)-\a_p(p^e-2)\bigr)\\
&=&\tfrac1{p-1}\bigl(\a_p(i)+(p-1)e-\a_p(i+1)-((p-1)e-1)\bigr)\\
&=&\tfrac1{p-1}\bigl((\a_p(i+1)-1+(p-1)\nu_p(i+1))\\
&&+(p-1)e-\a_p(i+1)-(p-1)e+1\bigr)\\
&=&\nu_p(i+1).\end{eqnarray*}
This equals $e-1$ when $i=c\cdot p^{e-1}-1$ for $1\le c\le p-1$ , and is less than $e-1$ for all other relevant $i$.
The result when $p=2$ follows since there is a single term of smallest exponent in the sum which defines $f(2^e-2)$.

Now let $p$ be odd.
We will show that for $1\le c\le p-1$
\begin{equation}\label{bc}\binom{p^e-2}{c\cdot p^{e-1}-1}/p^{e-1}\equiv(-1)^{c+1}c\mod p.\end{equation}
Then we obtain
$$f(p^e-2)={\frac1{p^{e-1}}}\bigl(\sum_{c=1}^{p-1}(-1)^{c+1}{\frac1c}+Ap)+\sum_{j<e-1}\frac1{p^j}A_j,$$
where $A,A_j\in\Z_p$. By a result of Eisenstein (\cite{Eis})
\begin{equation}\sum_{c=1}^{p-1}(-1)^{c+1}\frac1c\equiv \frac{2^p-2}p\mod p.\label{Eis}\end{equation}
If $p$ is not a Wieferich prime, then $\nu_p(\frac{2^p-2}p)=0$ and so $\nu_p(f(p^e-2))=-(e-1)$.

To prove (\ref{bc}), we note that the LHS equals
$$\frac{(p^e-2)\cdots(p^e-cp^{e-1})}{2\cdots(cp^{e-1}-1)p^{e-1}}\equiv \frac {-2}2\cdots\frac{-(cp^{e-1}-1)}{cp^{e-1}-1}\cdot(-c)$$
mod $p$. There are $cp^{e-1}-2$ of the fractions equal to $-1$, which when multiplied together and by $-c$ give the desired $(-1)^{c+1}c$.

If $p=1093$ or 3511, we consider $p^{e-1}f(p^e-2)$ mod $p^2$. This equals $$2p^{e-1}\sum_{i=1}^{(p^2-1)/2}\binom{p^e-2}{i\cdot p^{e-2}-1}^{-1}.$$
Similarly to the above argument for (\ref{bc}), we can show that
$$\frac{p^{e-1}}{\binom{p^e-2}{i\cdot p^{e-2}-1}}\equiv(-1)^{i-[\frac ip]+1}\frac p{i\binom{p-1}{[i/p]}}\mod p^2.$$
{\tt Maple} computes $$\sum_{i=1}^{(p^2-1)/2}(-1)^{i-[\frac ip]+1}\frac p{i\binom{p-1}{[i/p]}}\mod p^2$$
to equal $487\cdot1093$ when $p=1093$ and $51\cdot3511$ if $p=3511$. Thus $\nu_p(p^{e-1}f(p^e-2))=1$ when $p\in\{1093,3511\}$.
\end{proof}

\begin{proof}[Proof of Theorem \ref{2^e-k}] The recursive formula
\begin{equation}\label{rec1}f(n)=\frac{n+1}{2n}f(n-1)+1\end{equation}
was proved in \cite{Rock}. We invert it to obtain
\begin{equation}\label{rec2}f(n-1)=(f(n)-1)\frac{2n}{n+1}.\end{equation}

Our proof of Theorem \ref{2^e-k} is by induction on $k$, with Lemma \ref{2lem} being the case $k=1$.
Assume the theorem has been proved for $k-1$.
Let $$\nbar=\begin{cases}n&p=2\\1&p\text{ odd, not Wieferich}\\
2&p\in\{1093,3511\}.\end{cases}$$
 Then, with $\nu_p(u)=\nu_p(u')=0$,
$$f(p^e-k)-1=\frac{u\cdot p^{\nu_p(k-1)+\overline{k-1}}}{p^e}-1=u'\cdot p^{\nu_p(k-1)+\overline{k-1}-e},$$
since $e>\nu_p(k-1)+\overline{k-1}$. Now, using (\ref{rec2}),
$$f(p^e-k-1)=u'\cdot p^{\nu_p(k-1)+\overline{k-1}-e}\cdot\frac{2(p^e-k)}{p^e-k+1}=u''\cdot p^{\nu_p(k)+\overline{k}-e},$$
with $\nu_p(u'')=0$, since $\nu_p(p^e-k)=\nu_p(k)$ and $\nu_p(p^e-k+1)=\nu_p(k-1)$.
\end{proof}

The following proposition was used in the proof of \ref{-1} and will be used in the proof of \ref{0}.
\begin{prop}\label{p^e-1} For any odd prime $p$ and $e\ge 1$ and $1\le c\le p-1$,
$$f(cp^e-1)-f(cp^{e-1}-1)\equiv c(1-2^{p-1})p^{e-1}f(cp^{e-1}-1)\mod p^{e+1}.$$
\end{prop}
\begin{proof} 
For any $1\le j\le (p-1)/2$, we have
\begin{eqnarray*}&&
\binom{cp^e-1}{ip+2j-1}^{-1}+\binom{cp^e-1}{ip+2j}^{-1}=\binom{cp^e}{ip+2j}\bigg/\biggl(\binom{cp^e-1}{ip+2j-1}\binom{cp^e-1}{ip+2j}\biggr)\\
&=&\frac{cp^e}{ip+2j}\bigg/\binom{cp^e-1}{ip+2j}\equiv\frac{cp^e}{2j}\biggl/\biggl(\binom{cp^{e-1}-1}i\binom{p-1}{2j}\biggr)\mod p^{e+1}.
\end{eqnarray*}
The case $c=1$, $i=0$ says
$$\binom{p^e-1}{2j-1}^{-1}+\binom{p^e-1}{2j}^{-1}\equiv \frac{p^e}{2j}\biggl/\binom{p-1}{2j}\mod p^{e+1}.$$
Combining these, we obtain
$$\binom{cp^e-1}{ip+2j-1}^{-1}+\binom{cp^e-1}{ip+2j}^{-1}\equiv c\binom{cp^{e-1}-1}i^{-1}\bigl(\binom{p^e-1}{2j-1}^{-1}+\binom{p^e-1}{2j}^{-1}\bigr)\mod p^{e+1}.$$
Summing this over $j$ gives
\begin{eqnarray*}\sum_{k=1}^{p-1}\binom{cp^e-1}{ip+k}^{-1}&\equiv& c\binom{cp^{e-1}-1}i^{-1}\sum_{k=1}^{p-1}\binom{p^e-1}k^{-1}\\
&\equiv& c(1-2^{p-1})p^{e-1}\binom{cp^{e-1}-1}i^{-1}\mod p^{e+1},\end{eqnarray*}
using Lemma \ref{1-2} at the last step. Now sum over $i$ to obtain
\begin{equation}\label{part}\sum_{k\not\equiv0\ (p)}\binom{cp^e-1}{k}^{-1}\equiv c(1-2^{p-1})p^{e-1}f(cp^{e-1}-1)\mod p^{e+1}.\end{equation}

Summing $\binom{cp^e-1}{pj}^{-1}-\binom{cp^{e-1}-1}j^{-1}$ over $j$ gives
$$\sum_{k\equiv0\ (p)}\binom{cp^e-1}{k}^{-1}-f(cp^{e-1}-1)\equiv0\mod p^{e+2}$$
by Lemma \ref{newlem}.
Add this to (\ref{part}) to obtain the claim of the proposition.
\end{proof}

The above proof required the following lemmas.
\begin{lem}\label{newlem} If $1<u<p$ and $j<up^{e-1}$, and $b=ap+u$ with $a\ge0$, then
$$\nu_p\bigl(\tbinom{bp^e-1}{pj}^{-1}-\tbinom{bp^{e-1}-1}j^{-1}\bigr)\ge e+2.$$
\end{lem}
\begin{proof}
Note that
\begin{eqnarray*}&&\binom{bp^e-1}{pj}^{-1}-\binom{bp^{e-1}-1}j^{-1}\\
&=&\frac{(pj)(pj-1)\cdots1}{(bp^e-pj)(bp^e-(pj-1))\cdots(bp^e-1)}-\frac{j(j-1)\cdots1}{(bp^{e-1}-j)(bp^{e-1}-(j-1))\cdots(bp^{e-1}-1)}.\end{eqnarray*}
When the factors in the numerator and denominator of the second quotient are multiplied by $p$, they just give the $p$-divisible
factors of the first quotient. Thus the expression equals
$\ds\bigl(\prod \tfrac i{bp^e-i}-1\bigr)\cdot\tbinom{bp^{e-1}-1}j^{-1}$, where the product is taken over $0<i<pj$ with $i\not\equiv 0$ mod $p$.
Since $\nu_p\binom{bp^{e-1}-1}j=0$ and the number of values of $i$ is even, this has the same $p$-exponent as $\prod i-\prod (i-bp^e)$,
and this is divisible by $p^{e+2}$.

To see this, we will show that
$$\prod_{i=1}^{p-1}(pk+i)\equiv\prod_{i=1}^{p-1}(pk+i-bp^e)\mod p^{e+2}.$$
The difference, mod $p^{e+2}$, is, up to unit multiples,
\begin{eqnarray*}&&bp^e\sigma_{p-2}(pk+1,\ldots,pk+p-1)+(bp^e)^2\sigma_{p-3}(pk+1,\ldots,pk+p-1)\\
&=&bp^e\sigma_{p-2}(1,\ldots,p-1)+(bp^e\cdot pk+(bp^e)^2)\sigma_{p-3}(1,\ldots,p-1).\end{eqnarray*}
This is 0 since
 the coefficient of $x$ (resp.~$x^2$) in $(x+1)\cdots(x+p-1)$ is divisible by $p^2$ (resp.~$p$).
Indeed, the first is $(p-1)!(1+\cdots+\frac1{p-1})$, and this is divisible by $p^2$ for $p>3$ by \cite{Wol}.
Mod $p$, the polynomial equals $(x^2-1^2)(x^2-2^2)\cdots(x^2-(\frac{p-1}2)^2)$, and its coefficient of $x^2$ is congruent to
$1+\frac14+\cdots+\frac1{((p-1)/2)^2}$. Since $i^2\equiv(p-i)^2$ mod $p$, this is congruent to $\frac12\ds\sum_{i=1}^{p-1}\tfrac1{i^2}$,
and this is 0 mod $p$, also by \cite{Wol}.\end{proof}

\begin{lem}\label{1-2} For any odd prime $p$,
$$\sum_{k=1}^{p-1}\tbinom{p^e-1}k^{-1}\equiv p^{e-1}( 1-2^{p-1}) \mod p^{e+1}.$$
\end{lem}
\begin{proof} Noting that $\binom{p^e-1}i\equiv(-1)^i$ mod $p$, and arguing as in the proof of \ref{p^e-1}, we obtain
$$\sum_{k=1}^{p-1}\tbinom{p^e-1}k^{-1}=\sum_{j=1}^{(p-1)/2}\frac {p^e}{2j}\binom{p^e-1}{2j}^{-1}\equiv\frac {p^e}2\sum_{j=1}^{(p-1)/2}\frac 1j\mod p^2.$$
Next note that, mod $p$
$$\sum_{k=1}^{p-1}\frac{(-1)^k}k=\sum_{j=1}^{(p-1)/2}\bigl(\frac1{p-(2j-1)}-\frac1{2j-1}\bigr)\equiv\sum_{j=1}^{(p-1)/2}\frac{-1}{\frac{p-1}2+j}
\equiv\sum_{j=1}^{(p-1)/2}\frac 1j.$$
Using (\ref{Eis}), we obtain the
claim of the lemma.
\end{proof}

\section{Results for good primes and 23}\label{goodsec}
In this section, we prove (\ref{pdiff}) (and hence Theorem \ref{onlyodd}), deferring most details of the proof for $p=23$ to Section \ref{23sec}.
We will prove the following result in Section \ref{23sec}.
\begin{prop} \label{p^3} If $p$ is any odd prime, $1\le c\le p-1$, and $\nu_p(f(c-1))>0$, then
$$f(cp-1)-f(c-1)\equiv c(1-2^{p-1})f(c-1)\mod p^3.$$
\end{prop}
The following corollary follows easily from this and Proposition \ref{p^e-1}.

\begin{cor} If $p$ is an odd prime, $\nu_p(f(c-1))\le2$  and $e\ge0$, then $\nu_p(f(cp^e-1))=\nu_p(f(c-1))$.\label{0}\end{cor}
\begin{proof}  Since $1-2^{p-1}$ is divisible by $p$,  the result when $\nu_p(f(c-1))\le 1$ follows by induction on $e$ using Proposition \ref{p^e-1}.
The result when $\nu_p(f(c-1))=2$ follows in the same way, using Proposition \ref{p^3} for the first step.
\end{proof}

Now we can prove the following result, which implies (\ref{pdiff}) for good primes, since $\nu_p(i+1)\le e-1$ here.
\begin{prop} If $p$ is odd, $\nu_p(f(c-1))\le1$, $0\le i<p^e-1$, and $1\le c\le p-1$, then
\begin{equation}\nu_p(f(cp^e+i)-f(i))=-e+\nu_p(i+1)+\nu_p(f(c-1)).\label{3.2}\end{equation}\label{fce}
\end{prop}
\begin{proof} Let $e$ and $c$  be fixed, and let $\Delta(i):=f(cp^e+i)-f(i)$. The proof is by induction on $i$, using (\ref{rec1}). From (\ref{rec1}) we obtain
$$f(cp^e)-1=\frac{cp^e+1}{2cp^e}f(cp^e-1),$$
which, with Corollary \ref{0}, yields the claim for $i=0$. From (\ref{rec1}), we also deduce
\begin{equation}\label{ind}\Delta(i)=\frac{cp^e+i+1}{2(cp^e+i)}\Delta(i-1)-\frac{cp^e}{2i(cp^e+i)}f(i-1).\end{equation}
Assume the result for $i-1$. Then the first term in the RHS of (\ref{ind}) has exponent
$\nu_p(i+1)-\nu_p(i)+(-e+\nu_p(i)+\nu_p(f(c-1))$, which is the desired value. The exponent of the second term in the RHS of (\ref{ind}) is $$e-2\nu_p(i)+\nu_p(f(i-1))\ge e-2\nu_p(i)-\max\limits_j(\nu_p\tbinom{i-1}j).$$
We will know that the second term has larger exponent than the first once we
have shown  that, for $1\le i<p^e-1$ and $j\le i-1$,
$$2e-1>2\nu_p(i)+\nu_p(i+1)+\nu_p\tbinom{i-1}j .$$
This follows easily from the fact that $\nu_p\binom{i-1}j\le e-1-\nu_p(i)$, since it equals the number of carries in a base-$p$
addition whose sum is $i-1$.
\end{proof}

We will prove the following result in Section \ref{23sec}. The implication of this is that information about the case $e=1$ yields similar information for all $e$.
\begin{lem}\label{eind} If $1\le c,u\le p-1$ and $\nu_p(f(c-1))=2$,
then for all $e\ge 1$,
the expression $f(cp^e+up^{e-1}-1)-f(up^{e-1}-1)$ is divisible by $p$
and its residue  mod $p^2$ is independent of $e$.\end{lem}

Now we can prove the following analogue of Proposition \ref{fce}.
\begin{prop}\label{nu=2} Suppose $\nu_p(f(c-1))=2$ with $1\le c\le p-1$, and for all $u$ which satisfy $1\le u\le p-1$ and $\nu_p(f(u-1))=0$ we have
$\nu_p(f(cp+u-1)-f(u-1))=1$ and
\begin{equation}\label{23cond}\tfrac1p(f(cp+u-1)-f(u-1))\not\equiv\tfrac c{u}f(u-1)\mod p.\end{equation}
Then, for all $e\ge3$ and all $i$ satisfying $0\le i<p^e-1$, (\ref{3.2}) holds.\end{prop}

Suppose $p$ is a prime which is not good but has $\nu_p(f(c-1))\le2$ for all $1\le c\le p-1$ and whenever $\nu_p(f(c-1))=2$ the hypotheses of
Proposition \ref{nu=2} are satisfied. Then (\ref{3.2}) holds for $e\ge3$ and hence so does (\ref{pdiff}).  {\tt Maple} easily verifies that the conditions
of Proposition \ref{nu=2} are satisfied when $p=23$ and $c=13$. Hence (\ref{pdiff}), and thus also Theorem \ref{onlyodd}, holds when $p=23$.
Note that the proof of Theorem \ref{onlyodd} only cares about large values of $e$ in (\ref{pdiff}), and so our restriction to $e\ge3$ is not a problem.

\begin{proof}[Proof of Proposition \ref{nu=2}] As in the proof of Proposition \ref{fce}, fix $c$ and $e$, and let $\Delta(i)=f(cp^e+i)-f(i)$.
If we assume (\ref{3.2}) holds for $i-1$, the exponent of the first term of (\ref{ind}) equals the desired value of $\nu_p(\Delta(i))$. Now that $\nu(f(c-1))=2$,
it can happen that the second term has the same exponent if $i=up^{e-1}-1$ or $up^{e-1}$, in which case it could conceivably happen that the exponent
of the combination of the two terms is  larger than that of the individual terms.

In the next paragraph, we will use the hypothesis to prove directly that
$\nu_p(\Delta(up^{e-1}-1))$ and $\nu_p(\Delta(up^{e-1}))$ have the desired values when $1\le u\le p-1$. The validity for $\Delta(0)$ is verified
using (\ref{rec1}) as in the proof of \ref{fce}. For $0\le u\le p-1$, the induction from $up^{e-1}$ through $(u+1)p^{e-1}-2$ works just as it did
in the proof of \ref{fce}. Thus the result is valid for all $i<p^e-1$, as claimed.

That $\nu_p(\Delta(up^{e-1}-1))$ has the desired value 1 is immediate from Lemma \ref{eind} and the hypothesis that $\nu_p(\Delta(u-1))=1$. Mod $p$, (\ref{ind}) gives
\begin{eqnarray*}2up^{e-2}\Delta(up^{e-1})&\equiv&\tfrac1p\Delta(up^{e-1}-1)-\tfrac cu f(up^{e-1}-1)\\
&\equiv&\tfrac1p(f(cp+u-1)-f(u-1))-\tfrac cuf(u-1),\end{eqnarray*}
using Lemma \ref{eind} and Proposition \ref{p^e-1} at the last step. By Proposition \ref{nu=2}, the last expression has $\nu_p(-)=0$,
and so $\Delta(up^{e-1})$ has the claimed value of $-(e-2)$.
\end{proof}

\section{Proof of Theorem \ref{yes}}\label{2sec}
In this section, we prove Theorem \ref{yes}, which is a major difference between the situation when $p=2$ and the odd primes.
Our proof of Theorem \ref{yes} will use the following proposition.
\begin{prop}\label{sumrec} For $n\ge1$, $\ds \nu_2\bigl(\sum_{j=1}^n \tfrac1{2j-1}\bigr)=2\nu_2(n)$.\end{prop}
\begin{proof} The result when $n$ is odd follows from the result for $n-1$, since we are adding a number with $\nu_2=0$ to one with $\nu_2>0$.
We will use that
\begin{equation}\nu_2\bigl(\sum_{j=m+1}^n\tfrac1{2j-1}\bigr)=\nu_2(\sigma_{n-m-1}(2m+1,2m+3,\ldots,2n-1)),\label{sig}\end{equation}
where $\sigma(-)$ denotes an elementary symmetric polynomial, and its arguments are consecutive odd integers.
The following lemma about these will be useful. We will prove it after completing the proof of the proposition.
\begin{lem}\label{siglem} For $e\ge1$,
\begin{equation}\label{sig1}\sigma_{2^e-1}(-(2^e-1),-(2^e-3),\ldots,-1,1,\ldots,2^e-3,2^e-1)=0,\end{equation}
while for $e\ge2$,
\begin{equation}\label{sig2}\nu_2(\sigma_{2^e-2}(-(2^e-1),-(2^e-3),\ldots,-1,1,\ldots,2^e-3,2^e-1))=e-1.\end{equation}
\end{lem}

We first prove the proposition when $n=2^e$. Note that, mod $2^{2e+1}$,
\begin{eqnarray*}&&\sigma_{2^e-1}(2^e-(2^e-1),\ldots,2^e-1,2^e+1,\ldots,2^e+(2^e-1))\\
&\equiv&\sigma_{2^e-1}(-(2^e-1),\ldots,-1,1,\ldots,2^e-1)\\
&&+2\cdot2^e\sigma_{2^e-2}(-(2^e-1),\ldots,-1,1,\ldots,2^e-1)\\
&&+3\cdot2^{2e}\sigma_{2^e-3}(-(2^e-1),\ldots,-1,1,\ldots,2^e-1).\end{eqnarray*}
The factors of 2 and 3 occur since when $k$ factors are omitted, there are $k$ ways that the
first omission could have been chosen. The third term is 0 mod $2^{2e+1}$ since $\sigma_{2^e-3}(-(2^e-1),\ldots,-1,1,\ldots,2^e-1)$
is the sum of $\binom{2^e}{2^e-3}$ odd numbers, and $\binom{2^e}{2^e-3}$ is even.
By Lemma \ref{siglem}, the first of the three terms is 0 and the second has $\nu_2=2^{2e}$, implying the proposition when $n=2^e$.

We complete the proof by showing that validity for $n=2^e(2a-1)$ implies validity for $n=2^e(2a+1)$. This will be done
by showing
\begin{equation}\label{bignu}\nu_2\bigl(\frac1{2^{e+2}a-(2^{e+1}-1)}+\cdots+\frac1{2^{e+2}a+(2^{e+1}-1)}\bigr)>2e.\end{equation}
This is a sum of reciprocals of consecutive odd integers. The LHS is $\nu_2(\sigma)$, where, mod $2^{2e+4}$,
\begin{eqnarray*}\sigma&=&\sigma_{2^{e+1}-1}(2^{e+2}a-(2^{e+1}-1),\ldots,2^{e+2}a+ (2^{e+1}-1))\\
&\equiv&2\cdot2^{e+2}a\sigma_{2^{e+1}-2}(-(2^{e+1}-1),\ldots,-1,1,\ldots,2^{e+1}-1),\end{eqnarray*}
arguing similarly to the previous paragraph. By (\ref{sig2}), $\nu_2(\sigma)\ge e+3+e$.\end{proof}

\begin{proof}[Proof of Lemma \ref{siglem}] These are the coefficients of $x$ and $x^2$ in $\ds\prod_{j=1}^{2^{e-1}}(x^2-(2j-1)^2)$. Thus (\ref{sig1}) is clear,
and the LHS of (\ref{sig2}) equals $\nu_2(\sigma_{2^{e-1}-1}(1^2,3^2,\ldots,(2^e-1)^2)).$
We prove by induction on $e$ that this equals $e-1$. It is easily checked when $e=2$. Since all arguments are odd, we need
$$\nu_2\bigl(1+\frac1{3^2}+\cdots+\frac1{(2^e-1)^2}\bigr)=e-1.$$
Mod $2^e$, $1/j^2\equiv1/(2^e-j)^2$. Thus we need
$$\nu_2\bigl(1+\frac1{3^2}+\cdots+\frac1{(2^{e-1}-1)^2}\bigr)=e-2,$$
and this is the induction hypothesis.\end{proof}

Our next step toward the proof of Theorem \ref{yes} is
\begin{prop}\label{2^e-1} If $e\ge3$, then $\nu_2(f(2^e-1))\ge 2e$.\end{prop}

In fact, we conjecture that $\nu_2(f(2^e-1))=3e-2$ for $e\ge4$, but this seems a good bit harder to prove, and not much more useful
in proving something like Theorem \ref{yes}.

\begin{proof}[Proof of Proposition \ref{2^e-1}]
For $1\le j\le 2^{e-2}-1$, let
\begin{eqnarray*}p_{e,j}&:=&\tbinom{2^e-1}{2j}^{-1}+\tbinom{2^e-1}{2j-1}^{-1}+(-1)^{j+1}\bigl(\tbinom{2^{e-1}-1}j^{-1}+\tbinom{2^{e-1}-1}{j-1}^{-1}\bigr)\\
&=&\frac{(2j-1)!}{(2^e-1)\cdots(2^e-(2j-1))}\cdot\frac{2^e}{2^e-2j}\\
&&\qquad+(-1)^{j+1}\frac{(j-1)!}{(2^{e-1}-1)\cdots(2^{e-1}-(j-1))}\cdot\frac{2^{e-1}}{2^{e-1}-j}\\
&=&\frac{2^{e-1}}{2^{e-1}-j}\frac1{\binom{2^{e-1}-1}{j-1}}\bigl(\frac{(-1)^j}{(1-2^e)(1-\frac132^e)\cdots(1-\frac1{2j-1}2^e)}+(-1)^{j+1}\bigr)\\
&=&u\cdot2^{e-1-\nu_2(j)}\bigl(1+2^e(1+\tfrac13+\cdots+\tfrac1{2j-1})+2^{2e}A-1\bigr)\\
&=&u\cdot2^{2e-1-\nu_2(j)}(1+\tfrac13+\cdots+\tfrac1{2j-1}+2^{e}A).\end{eqnarray*}
Here $\nu_2(u)=0$, and $A\in\zt$ since it is a combination of elementary symmetric polynomials whose arguments are fractions with odd denominators.
The dots in the second (double) line range over all integers in the range, while in subsequent lines the dots range over odd integers in the range.
In going from the second line to the third, we have noted that the even factors in the first fraction, after dividing by 2, give the factors of the second fraction.
In going from the third line to the fourth, we have used that $\binom{2^{e-1}-1}{j-1}$ is odd.
Using Proposition \ref{sumrec} and that $e>\nu_2(j)$, we obtain $\nu_2(p_{e,j})\ge2e-1$.

Similarly we have
\begin{eqnarray*} p_{e,\text{mid}}&:=&\tbinom{2^e-1}{2^{e-1}-1}^{-1}-\tbinom{2^{e-1}-1}{2^{e-2}-1}^{-1}\\
&=&\frac1{\binom{2^{e-1}-1}{2^{e-2}-1}}\bigl(\frac1{(1-2^e)(1-\frac132^e)\cdots(1-\frac1{2^{e-1}-1}2^e)}-1\bigr)\\
&=&u'(2^e(1+\tfrac13+\cdots+\tfrac1{2^{e-1}-1})+2^{2e}A'),\end{eqnarray*}
satisfying $\nu_2(p_{e,\text{mid}})\ge 2e-1$.
Thus
$$\nu_2\bigl(\sum_{j=1}^{2^{e-2}-1}p_{e,j}+p_{e,\text{mid}}\bigr)\ge2e-1.$$

On the other hand,
$$\sum_{j=1}^{2^{e-2}-1}p_{e,j}=\sum_{i=1}^{2^{e-2}-2}\tbinom{2^e-1}i^{-1}+\tbinom{2^{e-1}-1}{2^{e-2}-1}^{-1}+\tbinom{2^{e-1}-1}0^{-1},$$
since most of the terms with alternating signs cancel.
Hence
$$\sum_{j=1}^{2^{e-2}-1}p_{e,j}+p_{e,\text{mid}}=\sum_{i=0}^{2^{e-1}-1}\tbinom{2^e-1}i^{-1},$$
which is $\frac12f(2^e-1)$. Thus $\nu_2(f(2^e-1))\ge2e$, as claimed.
\end{proof}

\begin{proof}[Proof of Theorem \ref{yes}] We will use Proposition \ref{2^e-1} and (\ref{rec1}) to prove for $i\ge0$ and $2^e>i$,
\begin{equation}\label{2^e+i} \nu_2(f(2^e+i)-f(i))\ge e-i-1.\end{equation}
Then, with $x$ as in Theorem \ref{yes}, let $E_k=\ds\sum_{i=0}^k2^{e_i}$. The hypothesis of the theorem and (\ref{2^e+i}) imply that for $k\ge N$,
$\nu_2(f(E_k)-f(E_{k-1}))\ge k$, and so $d(f(E_k),f(E_{k-1}))\le 2^{-k}$. Thus $\la f(E_k)\ra$ is a Cauchy sequence and so has a limit in $\Q_2$.
Thus $\fbar(x)$ is definable.

Now we prove (\ref{2^e+i}). Let $e$ be fixed, and $\Delta(i)=f(2^e+i)-f(i)$. Using Proposition \ref{2^e-1}, let $f(2^e-1)=A\cdot2^{2e}$ with $A\in\zt$. Using (\ref{rec1}) we obtain $\Delta(0)=A(2^e+1)2^{e-1}$ and
$$\Delta(i)=\frac{2^e+i+1}{2(2^e+i)}\Delta(i-1)-\frac{2^{e-1}}{(2^e+i)i}f(i-1).$$
Applying this iteratively, we obtain
\begin{equation}\label{induct}\Delta(i)=(2^e+i+1)(A2^{e-i-1}-\sum_{j=0}^{i-1}\frac{2^{e+j-i}f(j)}{(j+1)(2^e+j+2)(2^e+j+1)}).\end{equation}
Thus it suffices to prove $\nu_2(f(j))\ge-j-1+\nu_2((j+1)^2(j+2))$. This can be easily checked for $j\le3$. (It is sharp for $j=0$ and $j=2$.)

Since $\nu_2\binom ji\le[\log_2(j)]$, we deduce $\nu_2(f(j))\ge-[\log_2(j)]$. Since $j+1\ge[\log_2(j)]+2\nu_2(j+1)+\nu_2(j+2)$ for $j\ge4$, as is easily
proved, the desired result follows.
\end{proof}

Note how a comparison of (\ref{2^e+i}) and Proposition \ref{fce} points to a huge difference between the situations when $p$ is an even or odd prime.

\section{Proofs relevant to the case $\nu_p(f(c-1))=2$}\label{23sec}
In this section, we prove Proposition \ref{p^3} and Lemma \ref{eind}.
  The following lemmas will be useful in the proof of Proposition \ref{p^3}.
\begin{lem}\label{c-1} If $f(c-1)\equiv0\ (p)$, then $\ds\sum_{i=1}^{c-1} i\tbinom {c-1}i^{-1}\equiv0\ (p)$.\end{lem}
\begin{proof} We will prove the stronger result that $\sum i\binom {c-1}i^{-1}=\frac12(c-1)f(c-1)$. This is deduced from the following, where we use (\ref{rec1}) in the last step.
$$f(c-1)+\sum i\tbinom{c-1}i^{-1}=c\sum \tbinom c{i+1}^{-1}=c(f(c)-1)=c\tfrac{c+1}{2c}f(c-1).$$
\end{proof}
\begin{lem}\label{ABCD} If $0\le C\le A<p$ and $0\le D\le B<p$, then, mod $p^2$,
\begin{equation}\label{ABeq}\tbinom{Ap+B}{Cp+D}^{-1}-\tbinom AC^{-1}\tbinom BD^{-1}\equiv p\tbinom AC^{-1}\tbinom
BD^{-1}(C\sum_{i=1}^D{\tfrac1i}+(A-C)\sum_{i=1}^{B-D}{\tfrac1i}-A\sum_{i=1}^B{\tfrac1i}).\end{equation}
\end{lem}
\begin{proof} Since, for any $j$, $(pj+p-1)\cdots(pj+1)\equiv(p-1)!\mod {p^2}$, in evaluating
\begin{equation}\frac{(Cp+D)!((A-C)p+(B-D))!}{(Ap+B)!} \mod p^2,\label{B-D}\end{equation}
all the products which appear in between two multiples of $p$ cancel out. We cancel out a factor of $p$ from all multiples
of $p$ in the numerator and denominator of (\ref{B-D}) and obtain that (\ref{B-D}) is congruent to
$$\frac{C!(A-C)!}{A!}\cdot\frac{(Cp+D)\cdots(Cp+1)((A-C)p+(B-D))\cdots((A-C)p+1)}{(Ap+B)\cdots(Ap+1)}.$$
The second factor here equals $\frac{D!(B-D)!}{B!}\cdot E$, where
$$E=\frac{(\frac CDp+1)\cdots(\frac C1p+1)(\frac{A-C}{B-D}p+1)\cdots(\frac{A-C}1p+1)}{(\frac ABp+1)\cdots(\frac A1p+1)}.$$
The LHS of (\ref{ABeq}) is congruent, mod $p^2$, to $\binom AC^{-1}\binom BD^{-1}(E-1)$, and this is congruent to the claimed expression.
\end{proof}
\begin{proof}[Proof of Proposition \ref{p^3}] We are assuming that $\sum\binom{c-1}i^{-1}\equiv0$ mod $p$. We will prove that for $1\le j\le\frac{p-1}2$,
\begin{equation}\label{ip+2j} \sum_{i=0}^{c-1}\biggl(\frac1{ip+2j}\frac1{\binom{cp-1}{ip+2j}}-\frac1{2j\binom{c-1}i\binom{p-1}{2j}}\biggr)\equiv0\mod p^2.\end{equation}
Since
$$\binom{cp-1}{ip+2j-1}^{-1}+\binom{cp-1}{ip+2j}^{-1}=\frac{cp}{(ip+2j)\binom{cp-1}{ip+2j}},$$
when (\ref{ip+2j}) is summed over $j$ and multiplied by $cp$, we obtain
\begin{equation}\label{p3}\sum_{k\not\equiv0\ (p)}\binom{cp-1}k^{-1}\equiv c\sum_{k=1}^{p-1}\frac1{\binom{p-1}k}\sum_{i=0}^{c-1}\binom{c-1}i^{-1}\mod p^3.
\end{equation}
By Lemma \ref{newlem}
\begin{equation}\label{0modp^3}\binom{cp-1}{pj}^{-1}-\binom{c-1}j^{-1}\equiv0\mod p^3.\end{equation}
Adding (\ref{p3}) and the sum over $j$ of (\ref{0modp^3}) and using Lemma \ref{1-2} yields the claim of the proposition.

The proof of (\ref{ip+2j}) begins by noting that the LHS equals
\begin{equation}\label{LHS}\frac 1{cp-1}\sum_{i=0}^{c-1}\binom{cp-2}{ip+2j-1}^{-1}-\frac1{(p-1)\binom{p-2}{2j-1}}\sum_{i=0}^{c-1}\binom{c-1}i^{-1}.
\end{equation}
Since we are working mod $p^2$ and assuming $\sum\binom{c-1}i^{-1}\equiv0$ mod $p$, we may replace the $(p-1)$ in its coefficient by $-1$.
We will show
\begin{equation}\label{2} -\sum_{i=0}^{c-1}\binom{cp-2}{ip+2j-1}^{-1}+\binom{p-2}{2j-1}^{-1}\sum_{i=0}^{c-1}\binom{c-1}i^{-1}\equiv0\mod p^2.\end{equation}
This implies that $\ds\sum_{i=0}^{c-1}\tbinom{cp-2}{ip+2j-1}^{-1}\equiv0\ (p)$, and so whether its coefficient is $-1$ or $\frac1{cp-1}$ is irrelevant.
Thus (\ref{2}) implies (\ref{ip+2j}).

To prove (\ref{2}), we note that $-1$ times its LHS is a sum of terms of the form of (\ref{ABeq}) with $A=c-1$, $B=p-2$, $C=i$, and $D=2j-1$.
Let $S_m:=\ds\sum_{i=1}^m\tfrac 1i$. By Lemma \ref{ABCD}, the LHS of (\ref{2}) equals $-1$ times
\begin{eqnarray*}&&p{\tbinom{p-2}{2j-1}^{-1}}\sum_{i=0}^{c-1}\tbinom{c-1}i^{-1}\bigl(iS_{2j-1}+(c-1-i)S_{p-1-2j}-(c-1)S_{p-2}\bigr)\\
&=&p{\tbinom{p-2}{2j-1}^{-1}}\sum_{i=0}^{c-1}\tbinom{c-1}i^{-1}(Ki+L),\end{eqnarray*}
where $K$ and $L$ do not depend on $i$. This is 0 mod $p^2$, using Lemma \ref{c-1} and the assumption that $\nu_p(f(c-1))>0$.

\end{proof}

The remainder of the paper is devoted to the proof of Lemma \ref{eind}.
Let $D_e:=f(cp^e+up^{e-1}-1)-f(up^{e-1}-1)$. We will show that $\nu_p(D_{e+1}-D_e)\ge2$ for $e\ge1$. Using (\ref{ind}), one easily shows that
 $\nu_p(D_1)\ge1$. We deduce that $\nu_p(D_e)\ge1$ for all $e$.

 Let
$$L_e:=\sum_{i< up^{e-1}}\biggl(\binom{cp^e+up^{e-1}-1}i^{-1}-\binom{up^{e-1}-1}i^{-1}\biggr),\text{ and }H_e:=\sum_{i\ge up^{e-1}}\binom{cp^e+up^{e-1}-1}i^{-1}.$$
Then $D_e=L_e+H_e$. We will prove $\nu_p(L_{e+1}-L_e)\ge2$ and $\nu_p(H_{e+1}-H_e)\ge2$.

We have
\begin{eqnarray}&&\nonumber L_{e+1}-L_e\\
&=&\sum_{i=0}^{up^{e-1}-1}\biggl(\binom{cp^{e+1}+up^{e}-1}{pi}^{-1}-\binom{cp^e+up^{e-1}-1}i^{-1}\label{d1}\\
&&\qquad\qquad-\biggl(\binom{up^e-1}{pi}^{-1}-\binom{up^{e-1}-1}i^{-1}\biggr)\biggr)\label{d2}\\
&&+\sum_{k\not\equiv0\ (p)}^{k<up^e}\biggl(\binom{cp^{e+1}+up^{e}-1}k^{-1}-\binom{up^e-1}k^{-1}\biggr)\label{d3}.\end{eqnarray}
Summands of (\ref{d1}) and (\ref{d2}) are divisible by $p^2$ by Lemma \ref{newlem}.  Similarly to the beginning of the proof of \ref{p^e-1},
the sum (\ref{d3}) is a sum of terms of the form
\begin{equation}\label{d4}\frac{cp^{e+1}+up^e}{2j}\binom{cp^{e+1}+up^e-1}{2j}^{-1}-\frac{up^e}{2j}\binom{up^e-1}{2j}^{-1}\end{equation}
with $j\not\equiv0$ mod $p$. This is clearly divisible by $p^2$ when $e>1$. If $e=1$, the $cp^{e+1}$-part is 0 mod $p^2$, and $\binom{cp^{e+1}+up^e-1}{2j}\equiv\binom{up^e-1}{2j}$
mod $p$, which is good enough, since there is an additional factor of $p$. Thus (\ref{d4}) is 0 mod $p^2$ in this case, too.

Finally,
\begin{eqnarray}&&H_{e+1}-H_e\nonumber\\
&=&\sum_{i=up^{e-1}}^{cp^e+up^{e-1}-1}\binom{cp^{e+1}+up^e-1}{pi}^{-1}-\binom{cp^e+up^{e-1}-1}i^{-1}\label{h1}\\
&&+\sum_{k\not\equiv0\ (p)}^{k\ge up^e}\binom{cp^{e+1}+up^e-1}k^{-1}\label{h2}.\end{eqnarray}
The summands of (\ref{h1}) are divisible by $p^{e+1}$. To see this, we use the proof of Lemma \ref{newlem} but note that the factor
$\binom{cp^e+up^{e-1}-1}i^{-1}$ will have $\nu_p(-)=-1$.

Similarly to the proof of \ref{p^3}, (\ref{h2}) is a sum over $j$ of
$$\frac{cp^{e+1}+up^e}{cp^{e+1}+up^e-1}\sum_{i=up^{e-1}}^{cp^e+up^{e-1}-1}\binom{cp^{e+1}+up^e-2}{ip+2j-1}^{-1}.$$
Since the terms in the sum certainly have $\nu_p(-)\ge-1$, we are reduced to proving
\begin{equation}\label{f1}\sum_{i=u}^{cp+u-1}\binom{cp^2+up-2}{ip+2j-1}^{-1}\equiv0\mod p,\text{ and}\end{equation}
\begin{equation}\label{f2}\nu_p\biggl(\sum_{i=up }^{cp^2+up-1}\binom{cp^3+up^2-2}{ip+2j-1}^{-1}\biggr)\ge0.\end{equation}
We begin by proving (\ref{f1}).

We split the sum as
$$\sum_{a=1}^c\sum_{b=0}^{u-1}\binom{cp^2+(u-1)p+(p-2)}{ap^2+bp+2j-1}^{-1}+\sum_{a=0}^{c-1}\sum_{b=u}^{p-1}\binom{cp^2+(u-1)p+p-2}{ap^2+bp+2j-1}^{-1}.$$
The first sum, mod $p$, splits as $\ds\sum_{a=1}^c{\tbinom ca^{-1}}\cdot\sum_{b=0}^{u-1}{\tbinom{u-1}b^{-1}}\cdot\tbinom{p-2}{2j-1}^{-1}.$
The first factor is $f(c)-1=f(c-1)\frac{c+1}{2c}\equiv0$ mod $p^2$, and the other factors are $p$-integral.
With $S_m=\ds\sum_{i=1}^m\frac 1i$ and $\ell=2j-1$, we can prove, using methods similar to those in the proof of \ref{p^3}, that the term in the second sum above is congruent, mod $p$ to
$$\tbinom{cp+u-1}{ap+b}^{-1}\tbinom{p-2}{\ell}^{-1}(1+p(b(S_\ell-S_{p-2-\ell})+(u-1)(S_{p-2-\ell}-S_{p-2}))).$$
Thus the second sum becomes, mod $p$,
\begin{equation}\label{ab}\a\sum_{a=0}^{c-1}\sum_{b=u}^{p-1}{\tbinom {cp+u-1}{ap+b}^{-1}}+p\b\sum_{a=0}^{c-1}\sum_{b=u}^{p-1}b{\tbinom {cp+u-1}{ap+b}^{-1}}\end{equation}
with $\a$ and $\b$ $p$-integral.

We can prove that when $b\ge u$, mod $p^2$,
$$p\tbinom{cp+u-1}{ap+b}^{-1}\equiv\tfrac uc\tbinom{c-1}a^{-1}\tbinom bu\tbinom{p-1}{b-u}^{-1}(1+p(aS_b+(c-a-1)S_{u-b+p-1}-cS_{u-1})),$$
with $S_m$ as above. Thus $p$ times (\ref{ab}) is congruent, mod $p^2$, to
$$A\sum_{a=0}^{c-1}{\tbinom{c-1}a^{-1}}+Bp\sum_{a=0}^{c-1}a\tbinom{c-1}a^{-1},$$
with $A$ and $B$ $p$-integral. Since $\nu_p(f(c-1))=2$ and using Lemma \ref{c-1}, we obtain that this expression is divisible by $p^2$
and hence (\ref{ab}) is divisible by $p$, yielding the claim.

The proof of (\ref{f2}) is very similar. In fact, the sum here is $\equiv0$ mod $p$, stronger than required.
If we write the sum as
$$\sum_{i_0=0}^{p-1}\sum_{t=u}^{cp+u-1}\binom{cp^3+(u-1)p^2+(p^2-2)}{tp^2+(i_0p+\ell)}^{-1}$$
with $\ell=2j-1$ as before, then $\binom{p^2-2}{i_0p+\ell}$ behaves here as $\binom{p-2}{\ell}$ did before, and the $t$-sum is split
in the same way as the $i$-sum before.

\def\line{\rule{.6in}{.6pt}}

\end{document}